\documentclass{my_aims}

\usepackage{amsmath}
\usepackage{paralist}
\usepackage{graphics}
\usepackage{graphicx}
\usepackage{psfrag}

\usepackage[colorlinks=true]{hyperref}

\hypersetup{urlcolor=blue, citecolor=red}

\newtheorem{theorem}{Theorem}[section]
\newtheorem{corollary}[theorem]{Corollary}
\newtheorem{lemma}[theorem]{Lemma}
\theoremstyle{definition}
\newtheorem{definition}[theorem]{Definition}
\newtheorem{remark}[theorem]{Remark}
\newtheorem{example}[theorem]{Example}

\title[Euler-Lagrange equations on time scales]{Euler-Lagrange equations
for composition functionals in calculus of variations\\
on time scales}

\author[A.B. Malinowska and D.F.M. Torres]{}



\email{abmalinowska@ua.pt} \email{delfim@ua.pt}

\begin{document}

\maketitle

\centerline{\scshape Agnieszka B. Malinowska}
\medskip
{\footnotesize
 \centerline{Department of Mathematics, University of Aveiro}
   \centerline{3810-193 Aveiro, Portugal}
   \centerline{Faculty of Computer Science, Bia{\l}ystok University of Technology}
   \centerline{15-351 Bia\l ystok, Poland}
}

\medskip

\centerline{\scshape Delfim F. M. Torres }
\medskip
{\footnotesize
 \centerline{Department of Mathematics, University of Aveiro}
   \centerline{3810-193 Aveiro, Portugal}
}

\bigskip

\centerline{(Submitted 10-May-2009; revised 10-March-2010; accepted 04-July-2010)}

\begin{abstract}
In this paper we consider the problem of the calculus of variations
for a functional which is the composition of a certain scalar
function $H$ with the delta integral of a vector valued field $f$,
i.e., of the form $H
\left(\int_{a}^{b}f(t,x^{\sigma}(t),x^{\Delta}(t))\Delta t \right)$.
Euler-Lagrange equations, natural boundary conditions for such
problems as well as a necessary optimality condition for
isoperimetric problems, on a general time scale, are given. A number
of corollaries are obtained, and several examples illustrating the
new results are discussed in detail.
\end{abstract}


\section{Introduction}

The calculus on time scales was introduced by Bernd Aulbach and
Stefan Hilger in 1988 \cite{Au:Hilger}. The new theory bridges the
divide and extends the traditional areas of continuous and discrete
analysis and the various dialects of $q$-calculus \cite{dt:qc} into
a single theory \cite{BP,book:ts1,1st:book:ts}. The calculus of
variations on time scales was born with the works
\cite{cv:02,b7,Zeidan} and has interesting applications in Economics
\cite{Almeida:T,Atici:Uysal:06,Atici:Uysal:08,RuiDel07,Nat}.
Currently, several researchers are getting interested in the new
theory and contributing to its development (see, \textrm{e.g.},
\cite{ZbigDel,Bohner:F:T,bhn:Gus,Rui:Del:HO,Mal:Tor:09,Mal:Tor:Wei,AD:09,AD:10b,AD:10c}).
The present work is dedicated to the study of general
(non-classical) problems of calculus of variations on an arbitrary
time scale $\mathbb{T}$. As a particular case, by choosing
$\mathbb{T} = \mathbb{R}$, one gets the generalized calculus of
variations \cite{CLP} with functionals of the form
\begin{equation*}
H \left(\int_{a}^{b}f(t,x(t),x'(t))dt \right),
\end{equation*}
where $f$ has $n$ components and $H$ has $n$ independent variables.
Cases of calculus of variations as these appear in practical
applications (see \cite{CLP} and the references given therein) but
cannot be solved using the classical theory. Therefore, an extension
of this theory is needed.

The paper is organized as follows. In Section~\ref{sec:prm}, some
preliminaries on time scales are presented. Our results are given in
Section~\ref{sec:Euler} and Section~\ref{sec:Iso}. We begin
Section~\ref{sec:Euler} by formulating the general (non-classical)
problem of calculus of variations \eqref{vp} on an arbitrary time
scale. We obtain a general formula for the Euler-Lagrange equations
and natural boundary conditions for the general problem
(Theorem~\ref{thm:mr}), which are then applied to the product
(Corollary~\ref{cproduct}) and the quotient
(Corollary~\ref{cquotient}). In Section~\ref{sec:Iso} we prove a
necessary optimality condition for the general isoperimetric problem
(Theorem~\ref{th:iso} and Theorem~\ref{th:iso:abn}). Throughout the
paper several examples illustrating the new results are discussed in
detail.


\section{Preliminaries}
\label{sec:prm}

The following definitions and theorems will serve as a short
introduction to the calculus of time scales; they can be found in
\cite{BP,book:ts1}.

A nonempty closed subset of $\mathbb{R}$ is called a \emph{time
scale} and it is denoted by $\mathbb{T}$. The real numbers
$(\mathbb{R})$, the integers $(\mathbb{Z})$, the natural numbers
$(\mathbb{N})$, the $h$-numbers ($h\mathbb{Z}:=\{h z | z \in
\mathbb{Z}\}$, where $h>0$ is a fixed real number), and the
$q$-numbers ($q^{\mathbb{N}_0}:=\{q^k | k \in \mathbb{N}_0\}$, where
$q>1$ is a fixed real number) are examples of time scales, as are
$\{0,\frac{1}{2},1\}$, $[2,3]\cup \mathbb{N}$, and $[-1,1]\cup
[2,3]$, where $[-1,1]$ and $[2,3]$ are real number intervals. We
assume that a time scale $\mathbb{T}$ has the topology that it
inherits from the real numbers with the standard topology.

\begin{definition}
For $t\in\mathbb{T}$ we define he \emph{forward jump operator}
$\sigma:\mathbb{T}\rightarrow\mathbb{T}$ by
\begin{equation*}
\sigma(t)=\inf{\{s\in\mathbb{T}:s>t}\}, \mbox{ for all
$t\in\mathbb{T}$},
\end{equation*}
while the \emph{backward jump operator}
$\rho:\mathbb{T}\rightarrow\mathbb{T}$ is defined by
\begin{equation*}
\rho(t)=\sup{\{s\in\mathbb{T}:s<t}\},\mbox{ for all
$t\in\mathbb{T}$}.
\end{equation*}
\end{definition}

In this definition we consider $\sigma(M)=M$ if $\mathbb{T}$ has a
maximum $M$ and $\rho(m)=m$ if $\mathbb{T}$ has a minimum $m$.

A point $t\in\mathbb{T}$ is called \emph{right-dense},
\emph{right-scattered}, \emph{left-dense} and \emph{left-scattered}
if $\sigma(t)=t$, $\sigma(t)>t$, $\rho(t)=t$ and $\rho(t)<t$,
respectively. Points that are simultaneously right-scattered and
left-scattered are called \emph{isolated}. Points that are
simultaneously right-dense and left-dense are called \emph{dense}.

The \emph{graininess function} $\mu:\mathbb{T}\rightarrow[0,\infty)$
is defined by
\begin{equation*}
\mu(t)=\sigma(t)-t,\mbox{ for all $t\in\mathbb{T}$}.
\end{equation*}

\begin{example}
If $\mathbb{T} = \mathbb{R}$, then $\sigma(t) = \rho(t) = t$ and
$\mu(t)= 0$. If $\mathbb{T} = \mathbb{Z}$, then $\sigma(t) = t + 1$,
$\rho(t) = t - 1$, and $\mu(t)= 1$. On the other hand, if
$\mathbb{T} = q^{\mathbb{N}_0}$, where $q>1$ is a fixed real number,
then we have $\sigma(t) = q t$, $\rho(t) = q^{-1} t$, and $\mu(t)=
(q-1) t$.
 \end{example}

\begin{definition}
A time scale $\mathbb{T}$ is called \emph{regular} if the following
two conditions are satisfied:
\begin{itemize}
\item[(i)] $\sigma(\rho(t))=t$, for all $t\in \mathbb{T}$; and
\item[(ii)] $\rho(\sigma((t))=t$, for all $t\in \mathbb{T}$.
\end{itemize}
\end{definition}

Following \cite{BP}, let us define
$$\mathbb{T}^{\kappa}=\left\{\begin{array}{lcl}
\mathbb{T}\setminus(\rho(\sup\mathbb{T}),\sup \mathbb{T}] &\mbox{if}&\sup \mathbb{T}<\infty \\
\mathbb{T} &\mbox{if}&\sup \mathbb{T}=\infty.
\end{array}\right.$$

\begin{definition}
\label{def:de:dif} We say that a function
$f:\mathbb{T}\rightarrow\mathbb{R}$ is \emph{delta differentiable}
at $t\in\mathbb{T}^{\kappa}$ if there exists a number
$f^{\Delta}(t)$ such that for all $\varepsilon>0$ there is a
neighborhood $U$ of $t$ (\textrm{i.e.},
$U=(t-\delta,t+\delta)\cap\mathbb{T}$ for some $\delta>0$) such that
$$|f(\sigma(t))-f(s)-f^{\Delta}(t)(\sigma(t)-s)|
\leq\varepsilon|\sigma(t)-s|,\mbox{ for all $s\in U$}.$$ We call
$f^{\Delta}(t)$ the \emph{delta derivative} of $f$ at $t$ and $f$ is
said \emph{delta differentiable} on $\mathbb{T}^{\kappa}$ provided
$f^{\Delta}(t)$ exists for all $t\in\mathbb{T}^{\kappa}$.
\end{definition}

\begin{remark}
If $t \in \mathbb{T} \setminus \mathbb{T}^\kappa$, then
$f^{\Delta}(t)$ is not uniquely defined, since for such a point $t$,
small neighborhoods $U$ of $t$ consist only of $t$ and, besides, we
have $\sigma(t) = t$. For this reason, maximal left-scattered points
are omitted in Definition~\ref{def:de:dif}.
\end{remark}

Note that in right-dense points $f^{\Delta} (t)=lim_{s\rightarrow
t}\frac{f(t)-f(s)}{t-s}$, provided this limit exists, and in
right-scattered points $f^{\Delta}
(t)=\frac{f(\sigma(t))-f(t)}{\mu(t)}$, provided $f$ is continuous at
$t$.

\begin{example}
\label{ex:der:QC} If $\mathbb{T} = \mathbb{R}$, then $f^{\Delta}(t)
= f'(t)$, \textrm{i.e.}, the delta derivative coincides with the
usual one. If $\mathbb{T} = \mathbb{Z}$, then $f^{\Delta}(t) =
\Delta f(t) = f(t+1) - f(t)$. If $\mathbb{T} = q^{\mathbb{N}_0}$,
$q>1$, then $f^{\Delta} (t)=\frac{f(q t)-f(t)}{(q-1) t}$,
\textrm{i.e.}, we get the usual derivative of quantum calculus
\cite{QC}.
\end{example}

A function $f:\mathbb{T}\rightarrow\mathbb{R}$ is called
\emph{rd-continuous} if it is continuous at right-dense points and
if its left-sided limit exists at left-dense points. We denote the
set of all rd-continuous functions by C$_{\textrm{rd}}$ and the set
of all delta differentiable functions with rd-continuous derivative
by C$_{\textrm{rd}}^1$.

Now we introduce the concept of integral for a function
$f:\mathbb{T}\rightarrow\mathbb{R}$.

Let $a,b\in\mathbb{T}$ with $a\leq b$. We define the closed interval
$[a,b]$ in $\mathbb{T}$ by
\begin{equation*}
[a,b]:=\{t \in \mathbb{T}: a \leq t \leq b\}.
\end{equation*}
Open intervals and half-open intervals in $\mathbb{T}$ are defined
accordingly. In what follows all intervals will be time scale
intervals.

It is known that rd-continuous function possess an
\emph{antiderivative}, \textrm{i.e.}, there exists a function $F$
with $F^{\Delta}=f$, and in this case the \emph{delta integral} is
defined by
\begin{equation*}
\int_{a}^{b}f(t)\Delta t=F(b)-F(a)
\end{equation*}
for all $a,b\in\mathbb{T}$.

The delta integral has the following properties:
\begin{itemize}

\item[(i)] if $f\in C_{rd}$ and $t \in \mathbb{T}^{\kappa}$, then
\begin{equation*}
\int_t^{\sigma(t)}f(\tau)\Delta\tau=\mu(t)f(t)\, ;
\end{equation*}

\item[(ii)]if $a,b\in\mathbb{T}$ and $f,g\in C_{rd}$, then

\begin{equation*}
\int_{a}^{b}f(\sigma(t))g^{\Delta}(t)\Delta t
=\left[(fg)(t)\right]_{t=a}^{t=b}-\int_{a}^{b}f^{\Delta}(t)g(t)\Delta
t \, ,
\end{equation*}

\begin{equation*}
\int_{a}^{b}f(t)g^{\Delta}(t)\Delta t
=\left[(fg)(t)\right]_{t=a}^{t=b}-\int_{a}^{b}
f^{\Delta}(t)g(\sigma(t))\Delta t;
\end{equation*}

\item[(iii)] if $[a,b]$ consists of only isolated points, then

\begin{equation*}
\int_{a}^{b}f(t)\Delta t =\sum_{t\in[a,b)}\mu(t)f(t).
\end{equation*}
\end{itemize}

\begin{example}
Let $a, b \in \mathbb{T}$ with $a < b$. If $\mathbb{T} =
\mathbb{R}$, then $\int_{a}^{b}f(t)\Delta t = \int_{a}^{b}f(t) dt$,
where the integral on the right-hand side is the classical Riemann
integral. If $\mathbb{T} = \mathbb{Z}$, then $\int_{a}^{b}f(t)\Delta
t = \sum_{k=a}^{b-1} f(k)$. If $\mathbb{T} = q^{\mathbb{N}_0}$,
$q>1$, then $\int_{a}^{b}f(t)\Delta t = (1 - q) \sum_{t \in [a,b)} t
f(t)$.
\end{example}

The Dubois-Reymond lemma of the calculus of variations on time
scales will be useful for our purposes.

\begin{lemma}(Lemma of Dubois-Reymond \cite{b7})
\label{lemma:DR} Let $\mathbb{T}=[a,b]$ be a time scale with at
least three points and let $g\in C_{\textrm{rd}}$,
$g:\mathbb{T}^\kappa\rightarrow\mathbb{R}$. Then,
$$\int_{a}^{b}g(t) \cdot \eta^\Delta(t)\Delta t=0  \quad
\mbox{for all $\eta\in C_{\textrm{rd}}^1$ with
$\eta(a)=\eta(b)=0$}$$ if and only if $g(t)=c$ on
$\mathbb{T}^\kappa$ for some $c\in\mathbb{R}$.
\end{lemma}


\section{Euler-Lagrange equations}
\label{sec:Euler}

Let $\mathbb{T}$ be a time scale. Throughout we let $A,B\in
\mathbb{T}$ with $A<B$. For an interval $[c,d]\cap \mathbb{T}$ we
simply write $[c,d]$. We also abbreviate $f\circ\sigma$ by
$f^\sigma$. Now let $[a,b]$, with $a,b\in \mathbb{T}$ and $b<B$, be
a subinterval of $[A,B]$.

The general (non-classical) problem of the calculus of variations on
time scales under our consideration consists of minimizing or
maximizing a functional of the form
\begin{equation}
\label{vp}
\begin{gathered}
 \mathcal{L}[x]=H\left(\int_{a}^{b}f_{1}(t,x^{\sigma}(t),x^{\Delta}(t))\Delta
     t,\ldots, \int_{a}^{b}f_{n}(t,x^{\sigma}(t),x^{\Delta}(t))\Delta
     t\right),\\
     (x(a)=x_{a}), \quad (x(b)=x_{b})
 \end{gathered}
\end{equation}
over all $x\in C_{rd}^{1}$. Using parentheses around the end-point
conditions means that these conditions may or may not be present. We
assume that:
\begin{itemize}

\item[(i)] the function $H:\mathbb{R}^{n}\rightarrow \mathbb{R}$ has continuous
partial derivatives with respect to its arguments and we denote them
by $H'_{i}$, $i=1,\ldots,n$;

\item[(ii)] functions $(t,y,v)\rightarrow f_{i}(t,y,v)$ from $[a,b]\times \mathbb{R}^{2}$ to
$\mathbb{R}$, $i=1,\ldots,n$, have partial continuous derivatives with
respect to $y,v$ for all $t\in[a,b]$ and we denote them by $f_{iy}$,
$f_{iv}$;

\item [(iii)] $f_{i}$, $f_{iy}$,
$f_{iv}$, $i=1,\ldots,n$, are rd-continuous in $t$ for all $x\in
C_{rd}^{1}$.

\end{itemize}

A function $x\in C_{rd}^{1}$ is said to be an admissible function
provided that it satisfies the end-points conditions (if any is
given).

Let us consider the following norm in $C_{rd}^{1}$:
\begin{equation*}
   \|x\|_{1}=\sup_{t\in[a,b]}|x^{\sigma}(t)|+\sup_{t\in[a,b]}|x^{\Delta}(t)|.
\end{equation*}

\begin{definition}
An admissible function $\tilde{x}$ is said to be a \emph{weak local
minimizer} (respectively \emph{weak local maximizer}) for \eqref{vp}
if there exists $\delta >0$ such that $\mathcal{L}[\tilde{x}]
\leq \mathcal{L}[x]$ (respectively $\mathcal{L}[\tilde{x}] \geq \mathcal{L}[x]$)
for all admissible $x$ with $\|x-\tilde{x}\|_{1}<\delta$.
\end{definition}

Next theorem gives necessary optimality conditions for problem
\eqref{vp}.

\begin{theorem}
\label{thm:mr}
If $\tilde{x}$ is a weak local solution of the problem \eqref{vp},
then the Euler-Lagrange equation
\begin{equation}
\label{Euler}
\sum_{i=1}^{n}H'_{i}(\mathcal{F}_{1}[\tilde{x}],\ldots,
\mathcal{F}_{n}[\tilde{x}])\left(f_{iv}^{\Delta}(t,\tilde{x}^{\sigma}(t),\tilde{x}^{\Delta}(t))
-f_{iy}(t,\tilde{x}^{\sigma}(t),\tilde{x}^{\Delta}(t))\right)=0
\end{equation}
holds for all $t \in [a,b]^\kappa$, where
$\mathcal{F}_{i}[\tilde{x}]
=\int_{a}^{b}f_{i}(t,\tilde{x}^{\sigma}(t),\tilde{x}^{\Delta}(t))\Delta t$,
$i=1,\ldots,n$. Moreover, if $x(a)$ is not specified, then
\begin{equation}
\label{nat:l}
\sum_{i=1}^{n}H'_{i}(\mathcal{F}_{1}[\tilde{x}],\ldots,
\mathcal{F}_{n}[\tilde{x}])f_{iv}(a,\tilde{x}^{\sigma}(a),\tilde{x}^{\Delta}(a))=0;
\end{equation}
and if $x(b)$ is not specified, then
\begin{multline}
\label{nat:r}
\sum_{i=1}^{n}H'_{i}(\mathcal{F}_{1}[\tilde{x}],
\ldots,\mathcal{F}_{n}[\tilde{x}])\Bigl(f_{iv}(\rho(b),\tilde{x}^{\sigma}(\rho
(b)),\tilde{x}^{\Delta}(\rho(b)))\Bigr.\\
+\Bigl.\int_{\rho(b)}^bf_{iy}(t,\tilde{x}^{\sigma}(t),\tilde{x}^{\Delta}(t))\Delta
t\Bigr)=0.
\end{multline}
\end{theorem}

\begin{proof}
Suppose that $\mathcal{L}[x]$ has a weak local extremum at
$\tilde{x}$. For an admissible variation $h\in C_{rd}^{1}$ we define
a function $\phi:\mathbb{R}\rightarrow \mathbb{R}$ by
$\phi(\varepsilon) = \mathcal{L}[(\tilde{x} + \varepsilon h)] $. We
do not require $h(a)=0$ or $h(b)=0$ in case $x(a)$ or $x(b)$,
respectively, is free (it is possible that both are free). A
necessary condition for $\tilde{x}$ to be an extremizer for
$\mathcal{L}[x]$ is given by $\phi'(\varepsilon)|_{\varepsilon=0} =
0$. Using the chain rule for obtaining the derivative of a composed
function we get
\begin{equation*}
\phi'(\varepsilon)|_{\varepsilon=0}
=\sum_{i=1}^{n}H'_{i}(\mathcal{F}_{1}[\tilde{x}],\ldots,\mathcal{F}_{n}[\tilde{x}])
\int_a^b \left[ f_{iy}(\bullet) h^{\sigma}(t) + f_{iv}(\bullet)
h^{\Delta}(t)\right]\Delta t ,
\end{equation*}
where $(\bullet) =
\left(t,\tilde{x}^{\sigma}(t),\tilde{x}^{\Delta}(t)\right)$.
Integration by parts of the first term of the integrand gives
\begin{equation*}
\int_a^b f_{iy}(\bullet) h^{\sigma}(t) \Delta t =\int_a^t
f_{iy}(\circ) \Delta \tau h(t)|_{t=a}^{t=b}-\int_a^b \left(\int_a^t
f_{iy}(\circ)\Delta \tau h^{\Delta}(t)\right) \Delta t,
\end{equation*}
where $(\circ) =
\left(\tau,\tilde{x}^{\sigma}(\tau),\tilde{x}^{\Delta}(\tau)\right)$.The
necessary condition $\phi'(\varepsilon)|_{\varepsilon=0} = 0$ can be
written as
\begin{multline}
\label{eq:aft:IP} 0 = \int_a^b
h^{\Delta}(t)\sum_{i=1}^{n}H'_{i}(\mathcal{F}_{1}[\tilde{x}],\ldots,\mathcal{F}_{n}[\tilde{x}])
\left(f_{iv}(\bullet)-\int_a^t
f_{iy}(\circ)\Delta \tau \right)\Delta t \\
\left.+\sum_{i=1}^{n}H'_{i}(\mathcal{F}_{1}[\tilde{x}],\ldots,\mathcal{F}_{n}[\tilde{x}])
\left( \int_a^t f_{iy}(\circ)\Delta \tau
h(t)\right)\right|_{t=a}^{t=b}.
\end{multline}
 In particular, equation \eqref{eq:aft:IP} holds for all variations which are zero at both ends.
 For all such $h$'s the second term in \eqref{eq:aft:IP} is zero and by the
Dubois-Reymond Lemma~\ref{lemma:DR}, we have
\begin{equation}\label{eq:EL}
\sum_{i=1}^{n}H'_{i}(\mathcal{F}_{1}[\tilde{x}],\ldots,\mathcal{F}_{n}[\tilde{x}])
\left(f_{iv}(\bullet)-\int_a^t f_{iy}(\circ)\Delta \tau
\right)\Delta t=c,
\end{equation}
for some $c\in \mathbb{R}$ and all $t\in[a,b]$. Hence, equation
\eqref{Euler} holds for all $t \in [a,b]^\kappa$. Equation
\eqref{eq:aft:IP} must be satisfied for all admissible values of
$h(a)$ and $h(b)$. Consequently, equations \eqref{eq:aft:IP} and
\eqref{eq:EL} imply that
\begin{multline*}
0=\left(c+\sum_{i=1}^{n}H'_{i}(\mathcal{F}_{1}[\tilde{x}],\ldots,\mathcal{F}_{n}[\tilde{x}])\int_a^b
f_{iy}(\bullet)\Delta t\right)h(b)\\
-\left(c+\sum_{i=1}^{n}H'_{i}(\mathcal{F}_{1}[\tilde{x}],\ldots,\mathcal{F}_{n}[\tilde{x}])\int_a^a
f_{iy}(\bullet)\Delta t\right)h(a).
\end{multline*}
From the properties of the delta integral and from \eqref{eq:EL}, it
follows that
\begin{multline}
\label{eq:1}
0=h(b)\left\{\sum_{i=1}^{n}H'_{i}(\mathcal{F}_{1}[\tilde{x}],\ldots,
\mathcal{F}_{n}[\tilde{x}])\left(f_{iv}(\rho(b),\tilde{x}^{\sigma}(\rho
(b)),\tilde{x}^{\Delta}(\rho(b))\right.\right.\\
+\left.\left.\int_{\rho(b)}^bf_{iy}(t,\tilde{x}^{\sigma}(t),\tilde{x}^{\Delta}(t))\Delta
t\right)\right\}
-h(a)\left(\sum_{i=1}^{n}H'_{i}(\mathcal{F}_{1}[\tilde{x}],\ldots,
\mathcal{F}_{n}[\tilde{x}])f_{iv}(a,\tilde{x}^{\sigma}(a),\tilde{x}^{\Delta}(a))\right).
\end{multline}
If $x(t)$ is not preassigned at either end-point, then $h(a)$ and
$h(b)$ are both completely arbitrary and we conclude that their
coefficients in \eqref{eq:1} must each vanish. It follows that
condition \eqref{nat:l} holds when $x(a)$ is not given, and
condition \eqref{nat:r} holds when $x(b)$ is not given.
\end{proof}

\begin{remark}
\label{rem:mr:reg} Let $\mathbb{T}$ be a regular time scale. Then
from the properties of the delta integral we have
\begin{equation*}
\int_{\rho(b)}^b
f_{iy}(t,\tilde{x}^{\sigma}(t),\tilde{x}^{\Delta}(t))\Delta t=\mu
(\rho
(t))f_{iy}(\rho(b),\tilde{x}^{\sigma}(\rho(b)),\tilde{x}^{\Delta}(\rho(b))).
\end{equation*}
Therefore \eqref{nat:r} can be written in the form
\begin{equation*}
\begin{split}
\sum_{i=1}^{n}H'_{i}(\mathcal{F}_{1}[\tilde{x}],\ldots,
\mathcal{F}_{n}[\tilde{x}])\left\{f_{iv}(\rho(b),\tilde{x}^{\sigma}(\rho
(b)),\tilde{x}^{\Delta}(\rho(b)))\right.\\
+\left.\mu (\rho
(t))f_{iy}(\rho(b),\tilde{x}^{\sigma}(\rho(b)),\tilde{x}^{\Delta}(\rho(b)))\right\}.
\end{split}
\end{equation*}
\end{remark}

Choosing $\mathbb{T}=\mathbb{R}$ in Theorem~\ref{thm:mr} we
immediately obtain Theorem~3.1 and Equation~(4.1) in \cite{CLP}. The
Euler-Lagrange Equation for the product functional can be deduced
from Theorem~\ref{thm:mr}.

\begin{corollary}\label{cproduct}
If $\tilde{x}$ is a solution of the problem
\begin{equation*}
\begin{gathered}
\mathcal{L}[x]=\left(\int_{a}^{b}f_{1}(t,x^{\sigma}(t),x^{\Delta}(t))\Delta
     t\right)\left( \int_{a}^{b}f_{2}(t,x^{\sigma}(t),x^{\Delta}(t))\Delta
     t\right),\\
     (x(a)=x_{a}), \quad (x(b)=x_{b}),
\end{gathered}
\end{equation*}
then the Euler-Lagrange equation
\begin{multline*}
\mathcal{F}_{2}[\tilde{x}]\left(f_{1v}^{\Delta}(t,\tilde{x}^{\sigma}(t),\tilde{x}^{\Delta}(t))-
f_{1y}(t,\tilde{x}^{\sigma}(t),\tilde{x}^{\Delta}(t))\right)\\
+\mathcal{F}_{1}[\tilde{x}]\left(f_{2v}^{\Delta}(t,\tilde{x}^{\sigma}(t),\tilde{x}^{\Delta}(t))-
f_{2y}(t,\tilde{x}^{\sigma}(t),\tilde{x}^{\Delta}(t))\right)=0
\end{multline*}
holds for all $t \in [a,b]^\kappa$. Moreover, if $x(a)$ is not
specified, then
\begin{equation*}
\mathcal{F}_{2}[\tilde{x}]f_{1v}(a,\tilde{x}^{\sigma}(a),\tilde{x}^{\Delta}(a))
+\mathcal{F}_{1}[\tilde{x}]f_{2v}(a,\tilde{x}^{\sigma}(a),\tilde{x}^{\Delta}(a))=0;
\end{equation*}
if $x(b)$ is not specified, then
\begin{multline*}
\mathcal{F}_{2}[\tilde{x}]\left(f_{1v}(\rho(b),\tilde{x}^{\sigma}(\rho
(b)),\tilde{x}^{\Delta}(\rho(b)))+\int_{\rho(b)}^bf_{1y}(t,\tilde{x}^{\sigma}(t),\tilde{x}^{\Delta}(t))\Delta
t\right)\\
+\mathcal{F}_{1}[\tilde{x}]\left(f_{2v}(\rho(b),\tilde{x}^{\sigma}(\rho
(b)),\tilde{x}^{\Delta}(\rho(b)))+\int_{\rho(b)}^bf_{2y}(t,\tilde{x}^{\sigma}(t),\tilde{x}^{\Delta}(t))\Delta
t\right)=0.
\end{multline*}
\end{corollary}

\begin{remark}
In the particular case $\mathbb{T}=\mathbb{R}$,
Corollary~\ref{cproduct} gives Equation~(3.17) in \cite{CLP}.
\end{remark}

\begin{example}
Consider the problem
\begin{equation}\label{ex:product}
\begin{gathered}
\text{minimize} \quad
\mathcal{L}[x]=\left(\int_{0}^{1}(x^{\Delta}(t))^2\Delta
     t\right)\left(\int_{0}^{1}tx^{\Delta}(t) \Delta
     t\right)\\
     x(0)=0, \quad x(1)=1.
 \end{gathered}
\end{equation}
If $\tilde{x}$ is a local minimum of \eqref{ex:product}, then the
Euler-Lagrange equation must hold, i.e.,
\begin{equation}\label{ex:product:euler}
2\tilde{x}^{\Delta \Delta}(t)Q_{2}+Q_{1}=0,
\end{equation}
where
\begin{equation*}
Q_{1}=\mathcal{F}_{1}[\tilde{x}]=\int_{0}^{1}(\tilde{x}^{\Delta}(t))^2\Delta
t, \quad
Q_{2}=\mathcal{F}_{2}[\tilde{x}]=\int_{0}^{1}t\tilde{x}^{\Delta}(t)
\Delta t.
\end{equation*}
If  $Q_{2}= 0$, then also $Q_{1}=0$. This contradicts the fact that
on any time scale a global minimizer for the problem
\begin{equation*}
\begin{gathered}
\text{minimize} \quad
\mathcal{F}_{1}[x]=\int_{0}^{1}(x^{\Delta}(t))^2\Delta
     t\\
     x(0)=0, \quad x(1)=1
\end{gathered}
\end{equation*}
is $\bar{x}(t)=t$ and $\mathcal{F}_{1}[\bar{x}]=1$. Hence,
$Q_{2}\neq 0$ and equation \eqref{ex:product:euler} implies that
candidate solutions for problem \eqref{ex:product} are those
satisfying the delta differential equation
\begin{equation}\label{euler}
\tilde{x}^{\Delta \Delta}(t)=-\frac{Q_1}{2Q_2}
\end{equation}
subject to boundary conditions $x(0)=0$ and $x(1)=1$. Solving
equation \eqref{euler} we obtain
\begin{equation*}
x(t)=-\frac{Q_1}{2Q_2}\int_{0}^{t}\tau\Delta
\tau+1+\frac{Q_1}{2Q_2}\int_{0}^{1}\tau\Delta \tau.
\end{equation*}
Therefore, a solution of \eqref{euler} depends on the time scale.
Let us consider, for example, $\mathbb{T}=\mathbb{R}$ and
$\mathbb{T}=\left\{0,\frac{1}{2},1\right\}$. On
$\mathbb{T}=\mathbb{R}$ we obtain
\begin{equation}\label{sol:P}
x(t)=-\frac{Q_1}{4Q_2}t^2+\frac{4Q_2+Q_1}{4Q_2}t.
\end{equation}
Substituting \eqref{sol:P} into functionals $\mathcal{F}_1$ and
$\mathcal{F}_2$ gives
\begin{equation}\label{equation:Q1,Q2}
\begin{cases}
\frac{48Q_2^2+Q_1^2}{48Q_2^2}=Q_1\\
\frac{12Q_2-Q_1}{24Q_2}=Q_2.
\end{cases}
\end{equation}
Solving the system of equations \eqref{equation:Q1,Q2} we obtain
\begin{gather*}
\begin{cases}
Q_1=0\\
Q_2=0,
\end{cases}
\quad
\begin{cases}
Q_1=\frac{4}{3}\\
Q_2=\frac{1}{3}.
\end{cases}
\end{gather*}
Therefore,
\begin{equation*}
\tilde{x}(t)=-t^2+2t
\end{equation*}
is a candidate extremizer for problem \eqref{ex:product} on
$\mathbb{T}=\mathbb{R}$. Note that nothing can be concluded as to
whether $\tilde{x}$ gives a minimum, a maximum, or neither of these
for $\mathcal{L}$. \\
The solution of \eqref{euler} on
$\mathbb{T}=\left\{0,\frac{1}{2},1\right\}$ is
\begin{gather}\label{euler:T}
x(t)=
\begin{cases}
0 & \text{ if } t=0 \\
 \frac{1}{2}+\frac{Q_1}{16Q_2} & \text{ if } t=\frac{1}{2}\\
 1 & \text{ if } t=1.
\end{cases}
\end{gather}
Constants $Q_1$ and $Q_2$ are determined by substituting
\eqref{euler:T} into functionals $\mathcal{F}_1$ and
$\mathcal{F}_2$. The resulting
 system of equations is
\begin{equation}\label{equation:T}
    \begin{cases}
1+\frac{Q_1^2}{64Q_2^2}=Q_1\\
\frac{1}{4}-\frac{Q_1}{32Q_2}=Q_2.
\end{cases}
\end{equation}
Since system of equations \eqref{equation:T} has no real solutions,
we conclude that there exists no extremizer for problem
\eqref{ex:product} on $\mathbb{T}=\left\{0,\frac{1}{2},1\right\}$
among the set of functions that we consider to be admissible.
\end{example}

Assuming that the denominator does not vanish, the Euler-Lagrange
equation for the quotient problem can be deduced from
Theorem~\ref{thm:mr}.

\begin{corollary}
\label{cquotient}
If $\tilde{x}$ is a solution of the problem
\begin{equation*}
\begin{gathered}
\mathcal{L}[x]=\frac{\int_{a}^{b}f_{1}(t,x^{\sigma}(t),x^{\Delta}(t))\Delta
     t}{\int_{a}^{b}f_{2}(t,x^{\sigma}(t),x^{\Delta}(t))\Delta
     t},\\
     (x(a)=x_{a}), \quad (x(b)=x_{b}),
 \end{gathered}
\end{equation*}
then the Euler-Lagrange equation
\begin{multline*}
f_{1v}^{\Delta}(t,\tilde{x}^{\sigma}(t),\tilde{x}^{\Delta}(t))-
f_{1y}(t,\tilde{x}^{\sigma}(t),\tilde{x}^{\Delta}(t))\\
-Q\left(f_{2v}^{\Delta}(t,\tilde{x}^{\sigma}(t),\tilde{x}^{\Delta}(t))-
f_{2y}(t,\tilde{x}^{\sigma}(t),\tilde{x}^{\Delta}(t))\right)=0
\end{multline*}
holds for all $t \in [a,b]^\kappa$, where
$Q=\frac{\mathcal{F}_{1}[\tilde{x}]}{\mathcal{F}_{2}[\tilde{x}]}$.
Moreover, if $x(a)$ is not specified, then
\begin{equation*}
f_{1v}(a,\tilde{x}^{\sigma}(a),\tilde{x}^{\Delta}(a))
-Qf_{2v}(a,\tilde{x}^{\sigma}(a),\tilde{x}^{\Delta}(a))=0;
\end{equation*}
if $x(b)$ is not specified, then
\begin{multline*}
f_{1v}(\rho(b),\tilde{x}^{\sigma}(\rho
(b)),\tilde{x}^{\Delta}(\rho(b)))
+\int_{\rho(b)}^bf_{1y}(t,\tilde{x}^{\sigma}(t),\tilde{x}^{\Delta}(t))\Delta t\\
-Q\left(f_{2v}(\rho(b),\tilde{x}^{\sigma}(\rho
(b)),\tilde{x}^{\Delta}(\rho(b)))
+\int_{\rho(b)}^bf_{2y}(t,\tilde{x}^{\sigma}(t),\tilde{x}^{\Delta}(t))\Delta t\right)=0.
\end{multline*}
\end{corollary}

\begin{remark}
In the particular situation $\mathbb{T}=\mathbb{R}$,
Corollary~\ref{cquotient} gives Equation~(3.21) in \cite{CLP}.
\end{remark}

\begin{example}\label{ex:q:1}
Consider the problem
\begin{equation}
\label{ex:quotient:1}
\begin{gathered}
\text{minimize} \quad
\mathcal{L}[x]=\frac{\int_{0}^{2}(x^{\Delta}(t))^2 \Delta
     t}{\int_{0}^{2}(x^{\Delta}(t)+(x^{\Delta}(t))^2)\Delta
     t},\\
     x(0)=0, \quad x(2)=4.
\end{gathered}
\end{equation}
If $\tilde{x}$ is a local minimizer for \eqref{ex:quotient:1}, then
the Euler-Lagrange equation must hold, i.e.,
\begin{equation*}
0=[2\tilde{x}^{\Delta}(t)-Q(1+2\tilde{x}^{\Delta}(t))]^\Delta, \quad
t\in[0,2]^{\kappa},
\end{equation*}
where
\begin{equation*}
Q=\frac{\int_{0}^{2}(\tilde{x}^{\Delta}(t))^2 \Delta
     t}{\int_{0}^{2}(\tilde{x}^{\Delta}(t)+(\tilde{x}^{\Delta}(t))^2)\Delta
     t}.
\end{equation*}
Therefore,
\begin{equation*}
0=2\tilde{x}^{\Delta \Delta}(t)-Q2\tilde{x}^{\Delta \Delta}(t),
\quad t\in[0,2]^{\kappa}.
\end{equation*}
Thus $\tilde{x}^{\Delta \Delta}(t)=0$ or $Q=1$. The solution of the
delta differential equation
\begin{equation*}
    \begin{gathered}
x^{\Delta \Delta}(t)=0,\\
     x(0)=0, \quad x(2)=4
 \end{gathered}
\end{equation*}
does not depend on the time scale and it is $\tilde{x}(t)=2t$.
Observe that $\mathcal{L}[\tilde{x}]=\frac{2}{3}<1$. Therefore,
$\tilde{x}$ is a candidate local minimizer for problem
\eqref{ex:quotient:1}.
\end{example}

\begin{example}
\label{ex:q:2}
Consider the problem
\begin{equation}\label{ex:quotient:2}
\begin{gathered}
\text{extremize} \quad
\mathcal{L}[x]=\frac{\int_{0}^{1}tx^{\Delta}(t) \Delta
     t}{\int_{0}^{1}(x^{\Delta}(t))^2\Delta
     t},\\
     x(0)=0, \quad x(1)=1.
 \end{gathered}
\end{equation}
The Euler-Lagrange equation for this problem is
\begin{equation*}
0=1-2Qx^{\Delta \Delta}(t),
\end{equation*}
where $Q$ is the value of functional $\mathcal{L}$ in a solution of
\eqref{ex:quotient:2}. Since $Q\neq 0$, it follows that
\begin{equation}
\label{ex:2}
x^{\Delta \Delta}(t)=\frac{1}{2Q}.
\end{equation}
Solving equation \eqref{ex:2} we obtain
\begin{equation*}
x(t)=\frac{1}{2Q}\int_{0}^{t}\tau\Delta
\tau+1-\frac{1}{2Q}\int_{0}^{1}\tau\Delta \tau.
\end{equation*}
Therefore, a solution of \eqref{ex:2} depends on the time scale. Let
us consider, for example,  $\mathbb{T}=\mathbb{R}$ and
$\mathbb{T}=\{0,\frac{1}{2},1\}$. On $\mathbb{T}=\mathbb{R}$ we
obtain
\begin{equation}\label{sol:R}
x(t)=\frac{1}{4Q}t^2+\frac{4Q-1}{4Q}t.
\end{equation}
Substituting \eqref{sol:R} into functional $\mathcal{L}$ yields
\begin{equation}\label{equation:Q}
    \frac{24Q^2+2Q}{48Q^2+1}=Q.
\end{equation}
Solving equation \eqref{equation:Q} we obtain
$Q\in\left\{\frac{1}{4}-\frac{\sqrt{3}}{6},0,\frac{1}{4}+\frac{\sqrt{3}}{6}\right\}$.
Therefore,
\begin{equation*}
x_{1}(t)=\frac{3}{3-2\sqrt{3}}t^2+\frac{2\sqrt{3}}{2\sqrt{3}-3}t
\end{equation*}
is a candidate local minimizer while
\begin{equation*}
x_{2}(t)=\frac{3}{3+2\sqrt{3}}t^2+\frac{2\sqrt{3}}{2\sqrt{3}+3}t
\end{equation*}
is a candidate local maximizer for problem \eqref{ex:quotient:2} on
$\mathbb{T}=\mathbb{R}$. \\
The solution of \eqref{ex:2} on
$\mathbb{T}=\left\{0,\frac{1}{2},1\right\}$ is
\begin{gather}\label{sol1:T}
x(t)=
\begin{cases}
0 & \text{ if } t=0 \\
 \frac{1}{2}-\frac{1}{16Q} & \text{ if } t=\frac{1}{2}\\
 1 & \text{ if } t=1.
\end{cases}
\end{gather}
The constant $Q$ is determined by substituting  \eqref{sol1:T} into
$\mathcal{L}$. The resulting equation is
\begin{equation}\label{sol2:T}
    \frac{1}{4}+\frac{1}{32Q}=Q+\frac{1}{64Q} \, .
\end{equation}
Solving \eqref{sol2:T} we obtain $Q\in
\left\{\frac{1}{8}-\frac{\sqrt{2}}{8},\frac{1}{8}+\frac{\sqrt{2}}{8}\right\}$
and stationary functions are
\begin{gather}\label{sol:T:1}
x_{1}(t)=
\begin{cases}
0 & \text{ if } t=0 \\
 \frac{\sqrt{2}}{2\sqrt{2}-2} & \text{ if } t=\frac{1}{2}\\
 1 & \text{ if } t=1,
\end{cases}
\end{gather}
and \begin{gather}\label{sol:T:2}
 x_{2}(t)=
\begin{cases}
0 & \text{ if } t=0 \\
 \frac{\sqrt{2}}{2\sqrt{2}+2} & \text{ if } t=\frac{1}{2}\\
 1 & \text{ if } t=1.
\end{cases}
\end{gather}
\begin{figure}[ht]
\begin{minipage}[b]{0.49\linewidth}
\centering
\includegraphics[scale=0.26]{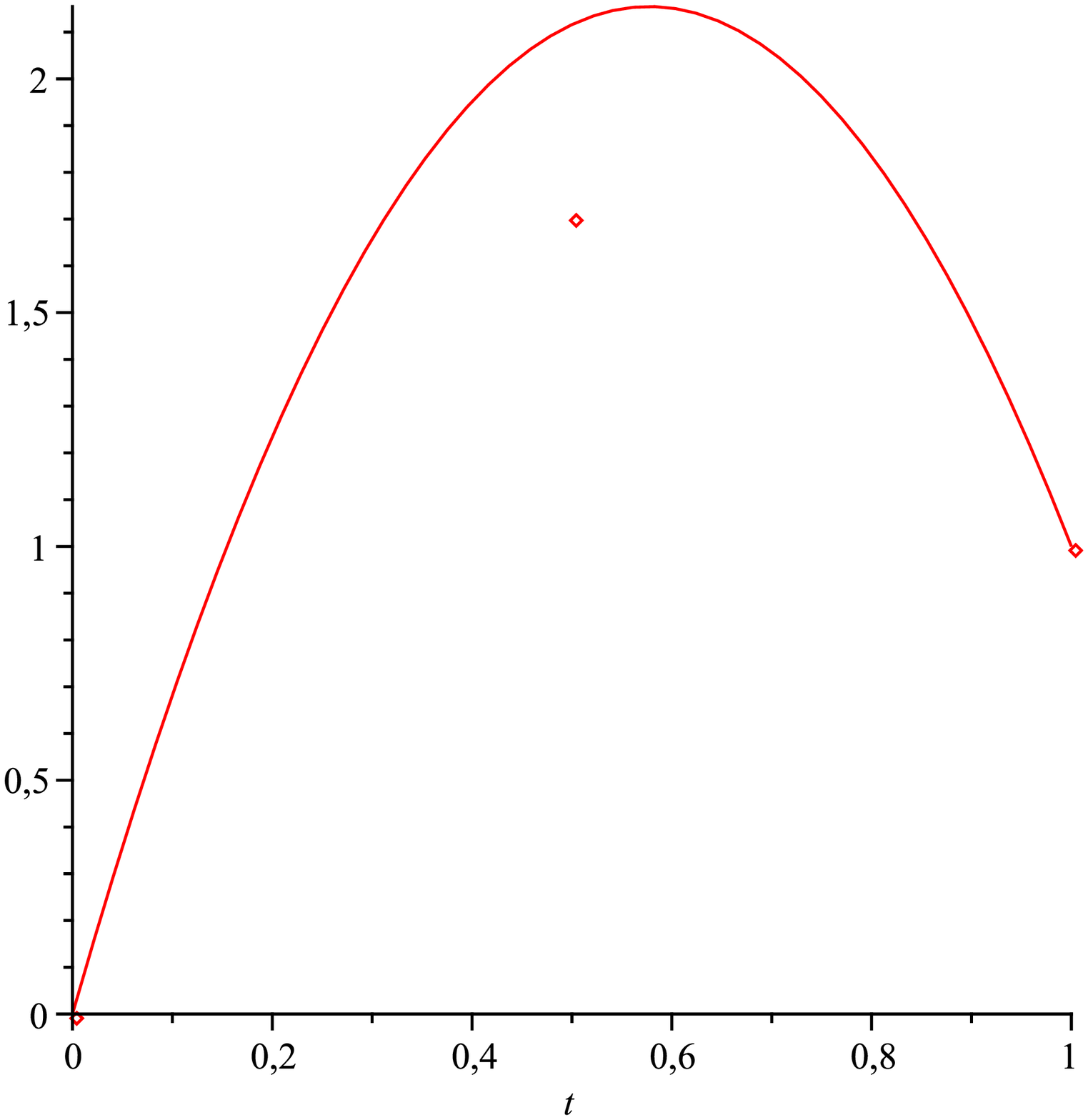}
\caption{The extremal minimizer of Example~\ref{ex:q:2} for
$\mathbb{T}=\mathbb{R}$ and $\mathbb{T}=\{0,\frac{1}{2},1\}$.}
\label{fig1:Ex6}
\end{minipage}
\begin{minipage}[b]{0.49\linewidth}
\centering
\includegraphics[scale=0.26]{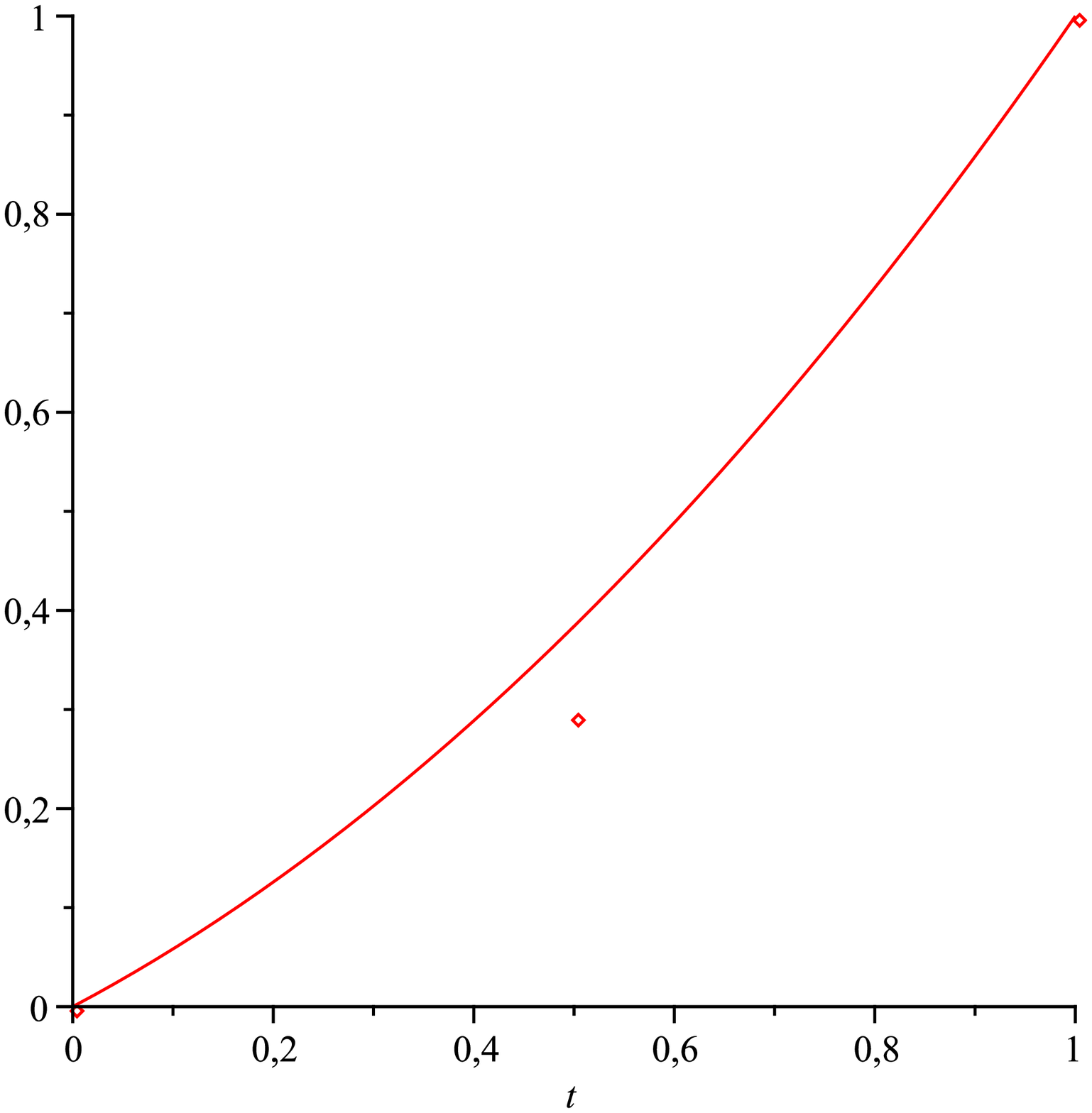}
\caption{The extremal maximizer of Example~\ref{ex:q:2} for for
$\mathbb{T}=\mathbb{R}$ and $\mathbb{T}=\{0,\frac{1}{2},1\}$.}
\label{fig2:Ex6}
\end{minipage}
\end{figure}

Therefore \eqref{sol:T:1} is a candidate local minimizer while
\eqref{sol:T:2} is a candidate local maximizer for problem
\eqref{ex:quotient:2} on
$\mathbb{T}=\left\{0,\frac{1}{2},1\right\}$.
\end{example}

\begin{example}\label{ex:SL}
Consider the problem
\begin{equation}
\label{ex:SL:1}
\begin{gathered}
\text{extremize} \quad
\mathcal{L}[x]=\frac{\int_{a}^{b}[(x^{\Delta}(t))^2-q(t)(x^{\sigma}(t))^2]
\Delta t}{\int_{a}^{b}(x^{\sigma}(t))^2\Delta
     t},\\
     x(a)=0, \quad x(b)=0,
 \end{gathered}
\end{equation}
where $q:[a,b]\rightarrow \mathbb{R}$ is a continuous function. The
Euler-Lagrange equation for this problem is
\begin{equation}\label{ex:SL:2}
x^{\Delta \Delta}(t)+q(t)x^{\sigma}(t)+Qx^{\sigma}(t)=0,
\end{equation}
subject to
\begin{equation}\label{ex:SL:3}
x(a)=0, \quad x(b)=0,
\end{equation}
where $Q$ is the value of functional $\mathcal{L}$ in a solution of
\eqref{ex:SL:1}. It is easily seen that
\eqref{ex:SL:2}--\eqref{ex:SL:3} is a case of the Sturm-Liouville
eigenvalue problem on time scales (see \cite{ABW} and
\cite{Rui:Del:ISO}). It follows that the problem of determining
eigenfunctions of \eqref{ex:SL:2} subject to \eqref{ex:SL:3} is
equivalent to the problem of determining functions satisfying
\eqref{ex:SL:3} which render $\mathcal{L}$ stationary.
\end{example}


\section{Isoperimetric problems}
 \label{sec:Iso}

Let us consider now the general (non-classical) isoperimetric
problem on time scales. The problem consists of minimizing or
maximizing
\begin{equation}\label{ivp}
 \mathcal{L}[x]=H\left(\int_{a}^{b}f_{1}(t,x^{\sigma}(t),x^{\Delta}(t))\Delta
     t,\ldots, \int_{a}^{b}f_{n}(t,x^{\sigma}(t),x^{\Delta}(t))\Delta
     t\right),
\end{equation}
in the class of functions $x\in C_{rd}^{1}$ satisfying the boundary
conditions
\begin{equation}\label{bivp}
x(a)=x_{a}, \quad x(b)=x_{b}
\end{equation}
and the constraint
\begin{equation}\label{civp}
\mathcal{K}[x]=P\left(\int_{a}^{b}g_{1}(t,x^{\sigma}(t),x^{\Delta}(t))\Delta
     t,\ldots, \int_{a}^{b}g_{m}(t,x^{\sigma}(t),x^{\Delta}(t))\Delta
     t\right)=k,
\end{equation}
where $x_{a}$, $x_{b}$, $k$ are given real numbers. We assume that:

\begin{itemize}
\item[(i)] functions $H:\mathbb{R}^{n}\rightarrow \mathbb{R}$
and $P:\mathbb{R}^{m}\rightarrow \mathbb{R}$ have continuous
partial derivatives with respect to their arguments and we denote
them by $H'_{i}$, $i=1,\ldots,n$, and $P'_{i}$, $i=1,\ldots,m$;

\item[(ii)] functions $(t,y,v)\rightarrow f_{i}(t,y,v)$, $i=1,\ldots,n$,
and $(t,y,v)\rightarrow g_{j}(t,y,v)$, $j=1,\ldots,m$,
from $[a,b]\times \mathbb{R}^{2}$ to $\mathbb{R}$
have partial continuous derivatives with respect to
$y,v$ for all $t\in[a,b]$ and we denote them by $f_{iy}$, $f_{iv}$
and $g_{jy}$, $g_{jv}$;

\item [(iii)] $f_{i}$, $f_{iy}$, $f_{iv}$, $i=1,\ldots,n$,
and $g_{j}$, $g_{jy}$, $g_{jv}$, $j=1,\ldots,m$,
are rd-continuous in $t$ for all $x\in C_{rd}^{1}$.
\end{itemize}

\begin{definition}
An admissible function $\tilde{x}$ is said to be a \emph{weak local
minimizer} (respectively \emph{weak local maximizer}) for the
isoperimetric problem \eqref{ivp}--\eqref{civp} if there exists
$\delta
>0$ such that $\mathcal{L}[\tilde{x}]\leq \mathcal{L}[x]$
(respectively $\mathcal{L}[\tilde{x}] \geq \mathcal{L}[x]$)
for all admissible $x$ satisfying the boundary conditions
\eqref{bivp}, the isoperimetric constraint \eqref{civp}, and
$\|x-\tilde{x}\|_{1}<\delta$.
\end{definition}

\begin{definition}
We say that $\tilde{x}$ is an extremal for $\mathcal{K}$ if
\begin{equation*}
\sum_{i=1}^{m}P'_{i}(\mathcal{G}_{1}[\tilde{x}],\ldots,\mathcal{G}_{m}[\tilde{x}])
\left(g_{iv}(\bullet)-\int_a^t g_{iy}(\circ)\Delta \tau \right)=c,
\end{equation*}
where $(\bullet) =
\left(t,\tilde{x}^{\sigma}(t),\tilde{x}^{\Delta}(t)\right)$ and
$(\circ) =
\left(\tau,\tilde{x}^{\sigma}(\tau),\tilde{x}^{\Delta}(\tau)\right)$,
for some constant $c$ and for all $t \in [a,b]$. An extremizer
(i.e., a weak local minimizer or a weak local maximizer) for the
problem \eqref{ivp}--\eqref{civp} that is not an extremal for
$\mathcal{K}$ is said to be a normal extremizer; otherwise (i.e., if
it is an extremal for $\mathcal{K}$), the extremizer is said to be
abnormal.
\end{definition}

\begin{theorem}
\label{th:iso}
If $\tilde{x}$ is a normal extremizer for the isoperimetric problem
\eqref{ivp}--\eqref{civp}, then there exists a real $\lambda$ such that
\begin{multline}\label{iso}
\sum_{i=1}^{n}H'_{i}(\mathcal{F}_{1}[\tilde{x}],\ldots,
\mathcal{F}_{n}[\tilde{x}])\left(f_{iv}^{\Delta}(t,\tilde{x}^{\sigma}(t),\tilde{x}^{\Delta}(t))
-f_{iy}(t,\tilde{x}^{\sigma}(t),\tilde{x}^{\Delta}(t))\right)\\
-\lambda\sum_{i=1}^{m}P'_{i}(\mathcal{G}_{1}[\tilde{x}],\ldots,
\mathcal{G}_{m}[\tilde{x}])\left(g_{iv}^{\Delta}(t,\tilde{x}^{\sigma}(t),\tilde{x}^{\Delta}(t))
- g_{iy}(t,\tilde{x}^{\sigma}(t),\tilde{x}^{\Delta}(t))\right)=0
\end{multline}
for all $t \in [a,b]^\kappa$.
\end{theorem}
\begin{proof}
Consider a variation of $\tilde{x}$, say
$\bar{x}=\tilde{x}+\varepsilon_{1}h_{1}+\varepsilon_{2}h_{2}$, where
$h_{i}\in C_{rd}^{1}$ and $h_{i}(a)=h_{i}(b)=0$, $i=1,2$, and
$\varepsilon_{i}$ is a sufficiently small parameter
($\varepsilon_{1}$ and $\varepsilon_{2}$ must be such that
$\|\bar{x}-\tilde{x}\|_{1}<\delta$ for some $\delta>0$). Here,
$h_{1}$ is an arbitrary fixed function and $h_{2}$ is a fixed
function that will be chosen later. Define the real function
\begin{equation*}
\bar{K}(\varepsilon_{1},\varepsilon_{2})=\mathcal{K}[\bar{x}]
=P\left(\int_{a}^{b}g_{1}(t,\bar{x}^{\sigma}(t),\bar{x}^{\Delta}(t))\Delta t,
\ldots, \int_{a}^{b}g_{m}(t,\bar{x}^{\sigma}(t),\bar{x}^{\Delta}(t))\Delta t\right)-k.
\end{equation*}
We have
\begin{equation*}
\left.\frac{\partial\bar{K}}{\partial
\varepsilon_{2}}\right|_{(0,0)}
=\sum_{i=1}^{m}P'_{i}(\mathcal{G}_{1}[\tilde{x}],\ldots,\mathcal{G}_{m}[\tilde{x}])
\int_a^b \left[ g_{iy}(\bullet) h_{2}^{\sigma}(t) + g_{iv}(\bullet)
h_{2}^{\Delta}(t)\right]\Delta t,
\end{equation*}
where $(\bullet) =
\left(t,\tilde{x}^{\sigma}(t),\tilde{x}^{\Delta}(t)\right)$. Since
$h_{2}(a)=h_{2}(b)=0$, integration by parts gives
\begin{equation*}
\int_{a}^{b}h_{2}^{\Delta}(t)\sum_{i=1}^{m}P'_{i}(\mathcal{G}_{1}[\tilde{x}],
\ldots,\mathcal{G}_{m}[\tilde{x}])
\left(g_{iv}(\bullet)-\int_a^t g_{iy}(\circ)\Delta \tau \right)\Delta t,
\end{equation*}
where $(\circ) =
\left(\tau,\tilde{x}^{\sigma}(\tau),\tilde{x}^{\Delta}(\tau)\right)$.
By Lemma~\ref{lemma:DR}, there exists $h_{2}$ such that
$\left.\frac{\partial\bar{K}}{\partial
\varepsilon_{2}}\right|_{(0,0)}\neq 0$. Since $\bar{K}(0,0)=0$, by
the implicit function theorem we conclude that there exists a
function $\varepsilon_{2}$ defined in the neighborhood of zero, such
that $\bar{K}(\varepsilon_{1},\varepsilon_{2}(\varepsilon_{1}))=0$,
i.e., we may choose a subset of variations $\bar{x}$ satisfying the
isoperimetric constraint. \\
Let us now consider the real function
\begin{equation*}
\bar{L}(\varepsilon_{1},\varepsilon_{2})=\mathcal{L}[\bar{x}]
=H\left(\int_{a}^{b}f_{1}(t,\bar{x}^{\sigma}(t),\bar{x}^{\Delta}(t))\Delta t,
\ldots, \int_{a}^{b}f_{n}(t,\bar{x}^{\sigma}(t),\bar{x}^{\Delta}(t))\Delta t\right).
\end{equation*}
By hypothesis, $(0,0)$ is an extremal of $\bar{L}$ subject to the
constraint $\bar{K}=0$ and $\nabla \bar{K}(0,0)\neq \textbf{0}$. By
the Lagrange multiplier rule, there exists some real $\lambda$ such
that $\nabla(\bar{L}(0,0)-\lambda\bar{K}(0,0))=\textbf{0}$. Having
in mind that $h_{1}(a)=h_{1}(b)=0$, we can write
\begin{equation*}
\left.\frac{\partial\bar{L}}{\partial
\varepsilon_{1}}\right|_{(0,0)}
=\int_{a}^{b}h_{1}^{\Delta}(t)\sum_{i=1}^{n}H'_{i}(\mathcal{F}_{1}[\tilde{x}],
\ldots,\mathcal{F}_{n}[\tilde{x}])
\left(f_{iv}(\bullet)-\int_a^t f_{iy}(\circ)\Delta \tau
\right)\Delta t
\end{equation*}
and
\begin{equation*}
\left.\frac{\partial\bar{K}}{\partial
\varepsilon_{1}}\right|_{(0,0)}
=\int_{a}^{b}h_{1}^{\Delta}(t)\sum_{i=1}^{m}P'_{i}(\mathcal{G}_{1}[\tilde{x}],
\ldots,\mathcal{G}_{m}[\tilde{x}])
\left(g_{iv}(\bullet)-\int_a^t g_{iy}(\circ)\Delta \tau
\right)\Delta t.
\end{equation*}
Therefore,
\begin{multline}\label{iso:3}
\int_{a}^{b}h_{1}^{\Delta}(t)\left\{\sum_{i=1}^{n}H'_{i}(\mathcal{F}_{1}[\tilde{x}],
\ldots,\mathcal{F}_{n}[\tilde{x}])
\left(f_{iv}(\bullet)-\int_a^t f_{iy}(\circ)\Delta \tau \right)\right.\\
-\left.\lambda \sum_{i=1}^{m}P'_{i}(\mathcal{G}_{1}[\tilde{x}],
\ldots,\mathcal{G}_{m}[\tilde{x}])
\left(g_{iv}(\bullet)-\int_a^t g_{iy}(\circ)\Delta \tau
\right)\right\}\Delta t=0.
\end{multline}
As \eqref{iso:3} holds for any $h_{1}$, by Lemma~\ref{lemma:DR}, we
have
\begin{multline}\label{iso:4}
\sum_{i=1}^{n}H'_{i}(\mathcal{F}_{1}[\tilde{x}],\ldots,\mathcal{F}_{n}[\tilde{x}])
\left(f_{iv}(\bullet)-\int_a^t f_{iy}(\circ)\Delta \tau \right)\\
-\lambda
\sum_{i=1}^{m}P'_{i}(\mathcal{G}_{1}[\tilde{x}],\ldots,\mathcal{G}_{m}[\tilde{x}])
\left(g_{iv}(\bullet)-\int_a^t g_{iy}(\circ)\Delta \tau \right)=c,
\end{multline}
for some $c\in \mathbb{R}$. Applying the delta derivative to both
sides of equation \eqref{iso:4}, we get \eqref{iso}.
\end{proof}

\begin{remark}
Choosing $H,P:\mathbb{R}\rightarrow \mathbb{R}$ and $H=P=id$ in
Theorem~\ref{th:iso} we immediately obtain Theorem~3.4 in
\cite{Rui:Del:ISO} and a particular case of Theorem~3.4 in
\cite{Mal:Tor:09}.
\end{remark}

One can easily cover abnormal extremizers within our result by
introducing an extra multiplier $\lambda_{0}$.

\begin{theorem}
\label{th:iso:abn}
If $\tilde{x}$ is an extremizer for the isoperimetric problem
\eqref{ivp}--\eqref{civp}, then there exist two constants
$\lambda_{0}$ and $\lambda$, not both zero, such that
\begin{multline}
\label{iso:abn}
\lambda_{0}\sum_{i=1}^{n}H'_{i}(\mathcal{F}_{1}[\tilde{x}],
\ldots,\mathcal{F}_{n}[\tilde{x}])\left(f_{iv}^{\Delta}(t,\tilde{x}^{\sigma}(t),\tilde{x}^{\Delta}(t))
-f_{iy}(t,\tilde{x}^{\sigma}(t),\tilde{x}^{\Delta}(t))\right)\\
-\lambda\sum_{i=1}^{m}P'_{i}(\mathcal{G}_{1}[\tilde{x}],\ldots,
\mathcal{G}_{m}[\tilde{x}])\left(g_{iv}^{\Delta}(t,\tilde{x}^{\sigma}(t),\tilde{x}^{\Delta}(t))
-g_{iy}(t,\tilde{x}^{\sigma}(t),\tilde{x}^{\Delta}(t))\right)=0
\end{multline}
for all $t \in [a,b]^\kappa$.
\end{theorem}
\begin{proof}
Following the proof of Theorem~\ref{th:iso}, since $(0,0)$ is an
extremal of $\bar{L}$ subject to the constraint $\bar{K}=0$, the
extended Lagrange multiplier rule (see, for instance,
\cite[Theorem~4.1.3]{Brunt}) asserts the existence of reals $\lambda_{0}$ and
$\lambda$, not both zero, such that
$\nabla(\lambda_{0}\bar{L}(0,0)-\lambda\bar{K}(0,0))=\textbf{0}$.
Therefore,
\begin{equation*}
\lambda_{0}\left.\frac{\partial\bar{L}}{\partial
\varepsilon_{1}}\right|_{(0,0)}-\lambda\left.\frac{\partial\bar{K}}{\partial
\varepsilon_{1}}\right|_{(0,0)}=0
\end{equation*}
\begin{multline}\label{abn:1}
\Leftrightarrow
\int_{a}^{b}h_{1}^{\Delta}(t)\left\{\lambda_{0}\sum_{i=1}^{n}H'_{i}(\mathcal{F}_{1}[\tilde{x}],
\ldots,\mathcal{F}_{n}[\tilde{x}])
\left(f_{iv}(\bullet)-\int_a^t f_{iy}(\circ)\Delta \tau \right)\right.\\
-\left.\lambda
\sum_{i=1}^{m}P'_{i}(\mathcal{G}_{1}[\tilde{x}],\ldots,\mathcal{G}_{m}[\tilde{x}])
\left(g_{iv}(\bullet)-\int_a^t g_{iy}(\circ)\Delta \tau
\right)\right\}\Delta t=0.
\end{multline}
Since \eqref{abn:1} holds for any $h_{1}$, it follows by
Lemma~\ref{lemma:DR} that
\begin{multline}\label{abn:2}
\lambda_{0}\sum_{i=1}^{n}H'_{i}(\mathcal{F}_{1}[\tilde{x}],\ldots,\mathcal{F}_{n}[\tilde{x}])
\left(f_{iv}(\bullet)-\int_a^t f_{iy}(\circ)\Delta \tau \right)\\
-\lambda
\sum_{i=1}^{m}P'_{i}(\mathcal{G}_{1}[\tilde{x}],\ldots,\mathcal{G}_{m}[\tilde{x}])
\left(g_{iv}(\bullet)-\int_a^t g_{iy}(\circ)\Delta \tau \right)=c,
\end{multline}
for some $c\in \mathbb{R}$. The desired condition \eqref{iso:abn}
follows by delta differentiation of \eqref{abn:2}.
\end{proof}

\begin{remark}
If $\tilde{x}$ is a normal extremizer for the isoperimetric problem
\eqref{ivp}--\eqref{civp}, then we can choose $\lambda_{0}=1$ in
Theorem~\ref{th:iso:abn} and obtain Theorem~\ref{th:iso}. For
abnormal extremizers, Theorem~\ref{th:iso:abn} holds with
$\lambda_{0}=0$. The condition $(\lambda_{0},\lambda)\neq\textbf{0}$
guarantees that Theorem~\ref{th:iso:abn} is a useful necessary
condition.
\end{remark}

\begin{corollary}
\label{c:iso:abn}
\begin{itemize}
\item[(i)] If $\tilde{x}$ is an extremizer for the isoperimetric
problem
\begin{equation*}
\begin{gathered}
\text{extremize}\quad
\mathcal{L}[x]=\left(\int_{a}^{b}f_{1}(t,x^{\sigma}(t),x^{\Delta}(t))\Delta
     t\right)\left( \int_{a}^{b}f_{2}(t,x^{\sigma}(t),x^{\Delta}(t))\Delta
     t)\right),\\
     x(a)=x_{a}, \quad x(b)=x_{b},
\end{gathered}
\end{equation*}
subject to the constraint
\begin{equation*}
\mathcal{K}[x]=\left(\int_{a}^{b}g_{1}(t,x^{\sigma}(t),x^{\Delta}(t))\Delta
     t\right)\left( \int_{a}^{b}g_{2}(t,x^{\sigma}(t),x^{\Delta}(t))\Delta
     t)\right)=k,
\end{equation*}
then there exist two constants $\lambda_{0}$ and $\lambda$, not both
zero, such that
\begin{equation*}
\lambda_{0}\left\{\mathcal{F}_{2}[\tilde{x}]\left(f_{1v}^{\Delta}(t,\tilde{x}^{\sigma}(t),\tilde{x}^{\Delta}(t))
-f_{1y}(t,\tilde{x}^{\sigma}(t),\tilde{x}^{\Delta}(t)\right)\right.
\end{equation*}
\begin{equation*}
\left.+\mathcal{F}_{1}[\tilde{x}]\left(f_{2v}^{\Delta}(t,\tilde{x}^{\sigma}(t),\tilde{x}^{\Delta}(t))
-f_{2y}(t,\tilde{x}^{\sigma}(t),\tilde{x}^{\Delta}(t))\right)\right\}
\end{equation*}
\begin{equation*}
-\lambda \left\{
\mathcal{G}_{2}[\tilde{x}]\left(g_{1v}^{\Delta}(t,\tilde{x}^{\sigma}(t),\tilde{x}^{\Delta}(t))
-g_{1y}(t,\tilde{x}^{\sigma}(t),\tilde{x}^{\Delta}(t)\right)\right.
\end{equation*}
\begin{equation*}
+\left.\mathcal{G}_{1}[\tilde{x}]\left(g_{2v}^{\Delta}(t,\tilde{x}^{\sigma}(t),\tilde{x}^{\Delta}(t))
-g_{2y}(t,\tilde{x}^{\sigma}(t),\tilde{x}^{\Delta}(t))\right)\right\}=0
\end{equation*}
for all $t \in [a,b]^\kappa$.

\item[(ii)] Assume that denominators of $\mathcal{L}$ and $\mathcal{K}$ do not vanish.
If $\tilde{x}$ is an extremizer for the isoperimetric problem
\begin{equation*}
\begin{split}
\text{extremize}\quad
\mathcal{L}[x]=\frac{\int_{a}^{b}f_{1}(t,x^{\sigma}(t),x^{\Delta}(t))\Delta
     t}{\int_{a}^{b}f_{2}(t,x^{\sigma}(t),x^{\Delta}(t))\Delta
     t}, \quad x(a)=x_{a}, \quad x(b)=x_{b},\\
 \end{split}
\end{equation*}
subject to the constraint
\begin{equation*}
\mathcal{K}[x]=\frac{\int_{a}^{b}g_{1}(t,x^{\sigma}(t),x^{\Delta}(t))\Delta
     t}{\int_{a}^{b}g_{2}(t,x^{\sigma}(t),x^{\Delta}(t))\Delta
     t}=k,
\end{equation*}
then there exist two constants $\lambda_{0}$ and $\lambda$, not both
zero, such that
\begin{equation*}
\lambda_{0}\left\{\mathcal{G}_{2}[\tilde{x}]\left(f_{1v}^{\Delta}(t,\tilde{x}^{\sigma}(t),\tilde{x}^{\Delta}(t))
-f_{1y}(t,\tilde{x}^{\sigma}(t),\tilde{x}^{\Delta}(t)\right)\right.
\end{equation*}
\begin{equation*}
\left.-\mathcal{G}_{2}[\tilde{x}]Q_{L}\left(f_{2v}^{\Delta}(t,\tilde{x}^{\sigma}(t),\tilde{x}^{\Delta}(t))
-f_{2y}(t,\tilde{x}^{\sigma}(t),\tilde{x}^{\Delta}(t))\right)\right\}
\end{equation*}
\begin{equation*}
-\lambda\left\{\mathcal{F}_{2}[\tilde{x}]\left(g_{1v}^{\Delta}(t,\tilde{x}^{\sigma}(t),\tilde{x}^{\Delta}(t))
-g_{1y}(t,\tilde{x}^{\sigma}(t),\tilde{x}^{\Delta}(t)\right)\right.
\end{equation*}
\begin{equation*}
\left.-\mathcal{F}_{2}[\tilde{x}]Q_{K}\left(g_{2v}^{\Delta}(t,\tilde{x}^{\sigma}(t),\tilde{x}^{\Delta}(t))
-g_{2y}(t,\tilde{x}^{\sigma}(t),\tilde{x}^{\Delta}(t))\right)\right\}=0
\end{equation*}
holds for all $t \in [a,b]^\kappa$, where
$Q_{L}=\frac{\mathcal{F}_{1}[\tilde{x}]}{\mathcal{F}_{2}[\tilde{x}]}$
and
$Q_{K}=\frac{\mathcal{G}_{1}[\tilde{x}]}{\mathcal{G}_{2}[\tilde{x}]}$.
\end{itemize}
\end{corollary}

\begin{example}
\label{ex:iso}
Consider the problem
\begin{equation}
\label{ex:iso:1}
\begin{split}
\text{extremize} \quad
\mathcal{L}[x]=\frac{\int_{0}^{1}(x^{\Delta}(t))^2 \Delta
     t}{\int_{0}^{1}tx^{\Delta}(t)\Delta
     t},\\
     x(0)=0, \quad x(1)=1,
 \end{split}
\end{equation}
subject to the constraint
\begin{equation}\label{ex:iso:2}
\mathcal{K}[x]=\int_{0}^{1}tx^{\Delta}(t)\Delta t=1.
\end{equation}
Since
\begin{equation*}
g_{v}(t,x^{\sigma}(t),x^{\Delta}(t))
-\int_{0}^{t}g_{y}(\tau,x_{\sigma}(\tau),x^{\Delta}(\tau))\Delta\tau=t
\end{equation*}
there are no abnormal extremals for the problem
\eqref{ex:iso:1}--\eqref{ex:iso:2}. Applying Theorem~\ref{th:iso},
we get the delta differential equation
\begin{equation}\label{rw}
2x^{\Delta\Delta}-Q-\lambda=0,
\end{equation}
where $Q$ is the value of functional $\mathcal{L}$ in a solution of
\eqref{ex:iso:1}--\eqref{ex:iso:2}. Solving equation \eqref{rw} we
obtain
\begin{equation*}
x(t)=\frac{Q+\lambda}{2}\int_{0}^{t}\tau\Delta
\tau+1-\frac{Q+\lambda}{2}\int_{0}^{1}\tau\Delta \tau.
\end{equation*}
Therefore, a solution of \eqref{rw} depends on the time scale. Let
us consider, for example, $\mathbb{T}=\mathbb{R}$ and
$\mathbb{T}=\{0,\frac{1}{2},1\}$. On $\mathbb{T}=\mathbb{R}$ we
obtain
\begin{equation*}
x(t)=3t^{2}-2t
\end{equation*}
as a candidate local minimizer while on
$\mathbb{T}=\{0,\frac{1}{2},1\}$
\begin{gather*}
x(t)=
\begin{cases}
0 & \text{ if } t=0 \\
-1 & \text{ if } t=\frac{1}{2}\\
 1 & \text{ if } t=1.
\end{cases}
\end{gather*}
is a candidate local minimizer for the problem
\eqref{ex:iso:1}--\eqref{ex:iso:2}.
\begin{figure}[ht]
\centering
\includegraphics[scale=0.4]{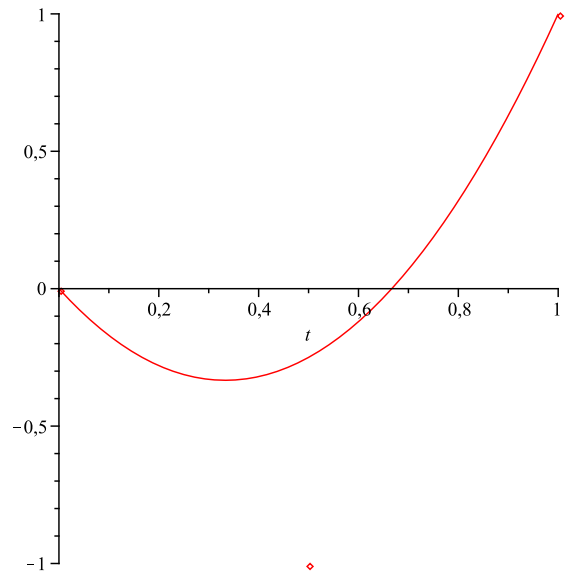}
\caption{The extremal minimizer of Example~\ref{ex:iso} for
$\mathbb{T}=\mathbb{R}$ and $\mathbb{T}=\{0,\frac{1}{2},1\}$.}
\label{fig:Ex7}
\end{figure}
\end{example}


\section*{Acknowledgements}

The authors were supported by the R\&D unit CEOC, via FCT and the EC
fund FEDER/POCI 2010. The first author was also supported by
Bia{\l}ystok University of Technology, via a project of the Polish
Ministry of Science and Higher Education ``Wsparcie miedzynarodowej
mobilnosci naukowcow''.

We are grateful to two anonymous reviewers for their comments.



\end{document}